\documentclass{amsart}
\usepackage{amsfonts,times,amssymb,amsmath,amsthm}
\usepackage{url}
\usepackage{color}
\usepackage{enumerate}

\renewcommand{\H}{{\Bbb H}}
\newcommand{\Z}{{\Bbb Z}}

\renewcommand{\leq}{\leqslant}
\renewcommand{\geq}{\geqslant}
\newcommand{\ql}{\operatorname{ql}}

\newcommand{\resp}{{\it resp.~}}
\renewcommand{\phi}{\varphi}
\newcommand{\prf}{\noindent{\bf Proof. }}
\newcommand{\nd}{\operatorname{ ndeg}}
\renewcommand{\prec}{\preccurlyeq}

\newtheorem{theorem}{Theorem}[section]
\newtheorem{lemma}{Lemma}[section]
\newtheorem{prop}{Proposition}[section]

\theoremstyle{definition}
\newtheorem{exa}{Example}[section]

\numberwithin{equation}{section}

\begin{document}

\title[]{Isotropy of 5-dimensional quadratic forms over the function field of a quadric in characteristic 2}

\thanks{
 The first author is supported by the {Deutsche Forschungsgemeinschaft} project \emph{The Pfister Factor Conjecture in characteristic two} (BE 2614/4) and the FWO Odysseus programme (project \emph{Explicit Methods in Quadratic Form Theory}).
 The second author acknowledges the support of the French Agence Nationale de la Recherche (ANR) under reference ANR-12-BL01-0005.}
\author{Andrew Dolphin}

\address{Departement Wiskunde--Informatica, Universiteit Antwerpen, Belgium}

\email{Andrew.Dolphin@uantwerpen.be}

\author{Ahmed Laghribi}

\address{Facult\'e des Sciences Jean Perrin, Laboratoire de math\'ematiques de Lens EA 2462, rue Jean Souvraz - SP18, 62307 Lens, France}

\email{ahmed.laghribi@univ-artois.fr}

\begin{abstract} We complete a classification of quadratic forms  over a field of characteristic $2$ of  type $(1,3)$ that become isotropic over the function field of a quadric.

\medskip\noindent
\emph{Keywords:}  Quadratic forms, function fields of quadrics, isotropy, characteristic two.

\medskip\noindent
\emph{Mathematics Subject Classification (MSC 2010):} 11E04; 11E81.

\end{abstract}

\maketitle

\section{Introduction}

Throughout this paper $F$ denotes a field of characteristic $2$. It is well-known, see \cite[(7.32)]{Elman:2008}, that any $F$-quadratic form $\phi$ is isometric to $$[a_1,b_1]\perp \cdots \perp [a_r,b_r]\perp \left< c_1, \cdots,c_s\right>$$ for  scalars $a_1,b_1, \cdots, a_r,b_r, c_1,\cdots, c_s\in F$, where $[a,b]$ (\resp $\left< c_1, \cdots, c_s\right>$) denotes the binary quadratic form given by $(x,y)\mapsto ax^2+xy+by^2$ (\resp the diagonal quadratic form given by $(x_1,\ldots, x_s)\mapsto\sum_{i=1}^sc_ix_i^2$). In this case the pair $(r,s)$ is unique, and we call it the type of $\phi$. The form $\left< c_1,\cdots, c_s\right>$ is also unique, we call it the quasilinear part of $\phi$ and we denote it by $\ql(\phi)$. We say that $\phi$ is nonsingular (\resp singular) if $s=0$ (\resp $s>0$) and totally singular if $r=0$. A notion that we will use in the formulation of our results is the domination relation between quadratic forms. Recall a quadratic form $\phi=(V,q)$ is called dominated by another quadratic form $\psi=(W, p)$, denoted  $\phi \prec \psi$, if there exists an injective $F$-linear map $f:V\longrightarrow W$ such that $q(v)=p(f(v))$ for every $v\in V$. The form $\phi$ is called weakly dominated by $\psi$, denoted  $\phi \prec_w \psi$, if $\alpha \phi \prec \psi$ for some $\alpha \in F^*$. 

Let $\phi$ be an anisotropic $F$-quadratic form. An important problem in the algebraic theory of quadratic forms is  classifying anisotropic $F$-quadratic forms $\psi$ for which $\phi$ becomes isotropic over $F(\psi)$, the function field of the affine quadric given by  $\psi$. This problem has been completely studied by the second author in \cite{l1} when $\phi$ is of dimension $\leq 4$,  of dimension $5$ and type $(2,1)$ or an Albert form (i.e., a 
nonsingular $6$-dimensional quadratic form of trivial Arf invariant). The  isotropy problem was treated by Faivre for certain  forms of dimension $6,7$ and $8$ in \cite{fairve:thesis} and  recently the second author  and Rehmann studied the isotropy of $5$-dimensional quadratic forms of type $(0,5)$ over function fields of quadrics in \cite{lr}. 
Our aim in this paper is to give a complete answer to the isotropy of   $5$-dimensional $F$-quadratic forms over the function field of a quadric
 for the remaining case, that is forms 
 of type $(1,3)$. 

An important  case where the isotropy question is well-known  concerns Pfister neighbours. More precisely, if $\phi$ is an anisotropic Pfister neighbour of a quadratic Pfister form $\pi$, then $\phi$ is isotropic over $F(\psi)$ if and only if $\pi$ is isotropic over $\psi$, which is equivalent, by Theorem \ref{trd}, to $\psi \prec_w \pi$.  In particular, in this case if the dimension of $\psi$ is greater than half the dimension of $\pi$, then $\psi$ is also a Pfister neighbour of $\pi$. 
Here, therefore, we will be most interested in the case where $\phi$ is not a Pfister neighbour.

In the following proposition we summarise several  cases of forms $\psi$ for which all forms  of type $(1,3)$ that are not Pfister neighbours remain anisotropic over $F(\psi)$. These  can be deduced from previous results proven by the second author, Hoffmann and Totaro.
Here $\triangle(\psi)$ denotes the Arf invariant of $\psi$ (see \S2).

\begin{prop}\label{ppar}
Let $\phi$ be an anisotropic $F$-quadratic form  of dimension $5$ and type $(1,3)$ that is not a Pfister neighbour, and let $\psi$ be anisotropic $F$-quadratic form  of  type $(r,s)$. Then $\phi$ is anisotropic over $F(\psi)$ in the following cases:
\begin{enumerate}[$(1)$]
\item $r=0$ and  $s=1$.
\item $r=0$ and $s\geq 5$.
\item $r=1$ and $s\geq 4$.
\item  $r=2$, $s=0$ and $\triangle(\psi)\neq 0$.
\item $r=2$ and $s\geq 1$.
\item $r\geq 3$.
\end{enumerate}
\end{prop}



The following results treat  quadratic forms $\psi$ where a quadratic form of type $(1,3)$ that is not a Pfister neighbour may become isotropic over $F(\psi)$. 
Recall that the norm degree of the totally singular $F$-quadratic  form  $\sigma$, denoted $\nd_F(\sigma)$, is the degree of the field $F^2(xy\mid x, y\in D_F(\sigma))$ over $F^2$, where $D_F(\sigma)$ is the set of scalars represented by $\sigma$. If $\sigma=\left<c_1, \cdots, c_s\right>$ with $c_1\neq 0$, then $\nd_F(\sigma)=[F^2(c_1c_2, c_1c_3, \cdots, c_1c_s):F^2]$ (see  \cite[\S8]{hl} for more details on the norm degree). In particular, $\nd_F(\sigma)$ is always a $2$-power and equal to or less than $2^{\dim\sigma}$.   Note that if $\sigma$ is of type $(0,4)$ and $\nd_F(\sigma)<4$, then $\psi$ is isotropic. We write   $GP_nF$ for the set of $F$-quadratic forms similar to $n$-fold quadratic Pfister forms and  $\varphi\sim \psi$ if $\varphi$ and $\psi$ are Witt equivalent $F$-quadratic forms (see \S2). 

\begin{theorem}\label{tp}
Let $\phi$ be an anisotropic $F$-quadratic form  of dimension $5$ and  type $(1,3)$ that is not a Pfister neighbour and let $\psi$ be anisotropic $F$-quadratic form. 
\begin{enumerate}[$(1)$]
\item If $\psi$ is of type $(1,1)$ or $(1,2)$, then $\phi$ is isotropic over $F(\psi)$ if and only if there exist $R_1, R_2$ nonsingular  $F$-quadratic forms  of dimension $2$, scalars $a,b, \alpha, \beta \in F^*$, a nonsingular completion $\rho$ of $\left<1,a,b\right>$ and $\pi\in GP_3F$ such that $\psi \prec_w \pi$, $\alpha \phi \cong R_1 \perp \left<1,a,b\right>$, $\beta\psi\cong R_2\perp Q$ and $R_1\perp R_2 \perp \rho \sim \pi$, where $Q=\left<1\right>$ or $\left<1,a\right>$ respectively as $\dim \psi=3$ or $4$.

\item If $\psi$ is of type $(2,0)$ and $\triangle(\psi)=0$ (that is, $\psi$  is similar to a $2$-fold Pfister form $\pi$), then $\phi_{F(\psi)}$ is isotropic if and only if $\phi_{F(\psi')}$ is isotropic, where $\psi'$ is a Pfister neighbour of $\pi$ of dimension $3$. Thus in this case we can reduce to case $(1)$.

\item If $\psi$ is of type $(1,3)$, then $\phi$ is isotropic over $F(\psi)$ if and only if $\phi$ is similar to $\psi$.

\item  If $\psi$ is of type $(0,3)$ or of type $(0,4)$ and $\nd_F(\psi)=8$,  then $\phi$ is isotropic over $F(\psi)$ if and only if  there exist  a form $\phi'$ of type $(1,3)$ and a form $\pi \in GP_3(F)$ such that $\phi \sim \phi' \perp \pi$, $\psi \prec_w \phi'$ and $\psi \prec_w \pi$.
\item If $\psi$ is of type $(0,4)$ and $\nd_F(\psi)=4$,  then $\phi_{F(\psi)}$ is isotropic if and only if $\phi_{F(\psi')}$ is isotropic, where $\psi'$ is a subform of $\psi$ of dimension $3$. Thus  in this case  we can  reduce to case $(4)$.
\end{enumerate}
\end{theorem}

The final case not included in Proposition \ref{ppar} and Theorem \ref{tp} is when $\psi$ is of type $(0,2)$. In this case, that any anisotropic $F$-quadratic form $\phi$ is isotropic over $F(\psi)$  if and only if $\psi\prec_w\phi$ is a classical result, but we include a proof in Lemma \ref{easy} for completeness.

The proofs of Proposition \ref{ppar} and Theorem \ref{tp} will be done case-by-case. For Proposition \ref{ppar} we will use some general results on the isotropy of quadratic forms of dimension $2^n+1$ over function fields of quadrics proved by the second author and Hoffmann. The proof of Theorem \ref{tp} is mainly based on the index reduction theorem in characteristic $2$  from \cite{ta} due to Mammone, Tignol and Wadsworth, and methods specific to totally singular quadratic forms.

\section{Definitions and preliminary results}

We recall the basic definitions and results we use from the theory of quadratic forms over fields. We refer to \cite{Elman:2008} as a general reference and for standard notation not explicitly defined here.

By an $F$-quadratic form  we  mean a pair $(V,q)$ of a finite dimensional $F$-vector space $V$ and a map  $q:V\rightarrow F$ such that
 $q(\lambda x)=\lambda^2q(x)$ for all $x\in V$ and $\lambda\in F$, and such that $b_q:V\times V\rightarrow F,(x,y)\longmapsto q(x+y)-q(x)-q(y)$ is $F$--bilinear. We call $b_q$ the polar form of $(V,q)$.
By an isometry of  $F$-quadratic  forms $\phi=(V,q)$ and $\psi=(W,p)$
   we mean an isomorphism of $F$--vector spaces $f:V\rightarrow W$ such that $q(x)=p(f(x))$ for all $x\in V$. If such an isometry exists, we say $\phi$ and $\psi$ are {isometric} and we write $\phi\simeq\psi$.
We say that $\phi$ is {similar to $\psi$} if there exists $c\in F^*$ such that $\phi\simeq c\psi$.

An $F$-quadratic form $\phi=(V,q)$ is called isotropic if there exists a nonzero vector $x\in V$ such that $q(x)=0$, otherwise $\phi$ is called anisotropic. We say that a scalar $\alpha \in F$ is represented by $\phi=(V,q)$ if there exists $x\in V$ such that $q(x)=\alpha$. The set of scalars represented by $\phi$ is denoted by $D_F(\phi)$. Any $F$-quadratic form $\phi$ has a unique decomposition $\phi \cong \phi_{an} \perp m\times [0,0] \perp n\times \left<0\right>$, where $m,n\geq 0$ are integers and $\phi_{an}$ is an anisotropic quadratic form uniquely determined up to isometry, which we call the anisotropic part of $\phi$. The integer $m$ (\resp $n$) is called the Witt index of $\phi$ and denoted $i_W(\phi)$ (\resp the defect index of $\phi$ and denoted $i_d(\phi)$).  
The form $\phi_{an}\perp m\times [0,0]$ is also unique.We call it the nondefective part of $\phi$ and we denote it by $\phi_{nd}$ (see \cite[(2.4)]{hl}). If $i_d(\phi)=0$ then we call $\phi$ nondefective. 

Two quadratic forms $\phi_1$ and $\phi_2$ are called Witt-equivalent and we write $\phi_1 \sim \phi_2$ if $\phi \perp m\times [0,0]\cong \phi_2 \perp n\times [0,0]$ for some integers $m,n\geq 0$.
Considering nonsingular quadratic forms up to Witt equivalence gives  the Witt group of nonsingular $F$-quadratic forms, $W_q(F)$. We let 
 $W(F)$ be the Witt ring of regular symmetric $F$-bilinear forms. There is a natural $W(F)$-module structure on $W_q(F)$ given by the tensor product of a symmetric bilinear form and a quadratic form (see \cite[p.51]{Elman:2008}).
Concerning  Witt cancellation, we recall the following result:

\begin{prop} (\cite[Prop.~1.2]{kn1} for $(1)$; \cite[Lem. 2.6]{hl} for $(2)$) Let $\mu$, $\nu$ be $F$-quadratic forms (possibly singular). Suppose that one of the two following conditions holds:
\begin{enumerate}[$(1)$] 
\item$\mu \perp \phi \cong \nu \perp \phi$ for some nonsingular form $\phi$.
\item $\mu$ and $\nu$ are nondefective and $\mu
\perp s\times\left< 0 \right> \cong \nu \perp s\times\left< 0
\right>$ for some integer $s$.\end{enumerate}
Then $\mu \cong \nu$.
\label{p3}
\end{prop}

Let $\phi=(V,q)$ be an $F$-quadratic form. Let $P_{\phi}$ be the homogeneous polynomial given by $\phi$ after a choice of an $F$-basis of $V$. The polynomial $P_{\phi}$ is reducible if and only if $\phi_{nd}$ is of type $(0,1)$ or $\phi_{nd}\cong [0,0]$, see \cite[Prop.~3]{ta}, and it is absolutely irreducible if $\phi$ is not totally singular and $\dim \phi_{nd}\geq 3$, see \cite{A}. When $P_{\phi}$ is irreducible we define the function field $F(\phi)$ of $\phi$ as the field of fractions of the quotient ring $$F\left[
x_1, \cdots, x_n\right]/(P_{\phi})\,.$$
We take $F(\phi)=F$ when $P_{\phi}$ is reducible or $\dim \phi=0$. 
If $P_{\phi}$ is absolutely irreducible and $K/F$ is a field extension, then the compositum $K\cdot F(\phi)$ coincides with $K(\phi)$. Note that if $\psi\prec \phi$ and $\dim \psi \geq 2$, then $\phi_{F(\psi)}$ is isotropic. Recall that if $\phi$ is nondefective, then $F(\phi)/F$ is transcendental if and only if $\phi$ is isotropic (see \cite[(22.9)]{Elman:2008}).

The following is well-known, but we include a proof for completeness. 
\begin{lemma}\label{easy}
Let $\phi$ be an anisotropic $F$-quadratic form and $\psi$ an anisotropic $F$-quadratic form over type $(0,2)$. Then $\phi_{F(\psi)}$ is isotropic if and only if $\psi \preccurlyeq_w \phi$.
\end{lemma}
\begin{proof} Let $\phi=(V,q)$. 
We may assume that $\psi=\left<1,d\right>$ for some $d\in F^*$. Then $\phi_{F(\psi)}$ is isotropic if and only if $\phi_{F(\sqrt{d})}$ is isotropic.  If $\phi _{F(\sqrt{d})}$ is isotropic
 then there exist  vectors $v, v' \in V\setminus\{0\}$ such that $q(v)=dq(v')$ and $b_{q}(v,v')=0$. Hence $\psi \preccurlyeq_w \phi$. The converse is clear.
\end{proof}

A quadratic form $\psi$ is called a subform of another quadratic form $\phi$, denoted by $\psi \subset \phi$, if there exists an $F$-quadratic form $\psi'$ such that $\phi \cong \psi \perp \psi'$. If  $\psi$ is nonsingular, then  the condition $\psi\prec \phi$ is equivalent to  $\psi\subset \phi$. The domination relation can be viewed as follows:

\begin{prop} (\cite[Lem. 3.1]{hl}) \label{nsc} Let $\phi$ and $\psi$ be $F$-quadratic forms. Then $\psi \preccurlyeq \phi$ if and only if there exist nonsingular forms $\psi_r$ and $\rho$, nonnegative integers $s'\leq s\leq s''$, $c_i\in F$ for $i\leq s''$, and $d_j\in F$ for $j\leq s'$ such that:
$$
\psi\simeq \psi_r \perp \left< c_1, \cdots, c_s\right>\,,$$
$$\phi\cong \psi_r \perp \rho \perp [c_1,d_1]\perp \cdots \perp [c_{s'},d_{s'}]\perp \left<c_{s'+1}, \cdots, c_{s''}\right>\,.$$
\end{prop}

The subform theorem will be also needed in our proof of Theorem \ref{tp}:
\begin{theorem} (\cite[Th. 4.2]{hl}) 
Let $\phi$ and $\psi$ be $F$-quadratic forms such that $\phi$ is anisotropic and nonsingular and $\psi$ is nondefective. If $\phi_{F(\psi)}$ is hyperbolic then $ab\psi \prec \phi$ for any $a\in D_F(\phi)$ and $b\in D_F(\psi)$.
\label{trd}
\end{theorem}

A nonsingular completion  of a totally singular $F$-quadratic form $\sigma=\left<c_1, \cdots, c_s\right>$ is nonsingular $F$-quadratic form isometric to $[c_1,d_1]\perp \cdots \perp [c_s,d_s]$ for some scalars $d_1, \cdots, d_s\in F$. Note that for any nonsingular completion $\rho$ of $\sigma$, we have $\rho \perp \sigma \sim \sigma$ because $[c,d]\perp \left<c\right> \cong [0,0]\perp \left<c\right>$ for any $c,d\in F$. 

Another fact related to the domination relation that we will use is the following result knows as the ``Completion Lemma'':

\begin{prop} (\cite[Lem.~3.9]{hl}) Let $\phi$ and $\psi$ be nonsingular $F$-quadratic forms and $c_1, \cdots, c_s\in F$ such that $\phi \perp \left<c_1, \cdots, c_s\right> \cong \psi \perp \left<c_1, \cdots, c_s\right>$. For any nonsingular completion $\rho$ of $\left<c_1, \cdots, c_s\right>$, there exists a nonsingular completion $\rho'$ of $\left<c_1, \cdots, c_s\right>$ such that $\phi \perp \rho\cong \psi \perp \rho'$.
\label{p6}
\end{prop}

For $n\in \mathbb{N}$, $n>0$ and  $a_1, \cdots, a_n\in F^*$, let $\left<a_1, \dots, a_n\right>_b$ denote the $n$-dimensional symmetric bilinear form given by $((x_1,\ldots, x_n), (y_1,\ldots, y_n))\mapsto \sum_{i=1}^na_ix_iy_i$. A bilinear form isometric to $\left< 1, a_1\right>_b\otimes \ldots \otimes \left< 1,a_n\right>_b$ is called an $n$-fold bilinear Pfister form and denoted by $\left<\left< a_1, \cdots, a_n\right>\right>_b$. 
By a $0$-fold bilinear Pfister form, we mean the form $\left<1\right>_b$.
For $n\in \mathbb{N}$, $n>0$, an $(n+1)$-fold quadratic Pfister form (or simply just an $(n+1)$-fold Pfister form) is a quadratic form isometric to 
the tensor product of an $n$-fold bilinear Pfister form and a nonsingular quadratic form representing $1$,  where the tensor product  is the $W(F)$-module action on $W_q(F)$. 
Let $P_n(F)$ (\resp $GP_n(F)$)  denote the set of $n$-fold quadratic Pfister forms (\resp the set $\{\alpha \pi\mid \alpha \in F^*\; \text{and}\; \pi \in P_n(F)\}$). Recall that a quadratic Pfister form is hyperbolic if it is isotropic (see \cite[(9.10)]{Elman:2008}).

An $F$-quadratic form $\phi$ is called a Pfister neighbour if there exists a quadratic Pfister form $\pi$ such that $2\dim \phi>\dim \pi$ and $\phi \prec_w \pi$. In this case the form $\pi$ is unique, and for any field extension $K/F$, the form $\phi_K$ is isotropic if and only if $\pi_K$ is isotropic. In particular, the forms $\phi_{F(\pi)}$ and $\pi_{F(\phi)}$ are isotropic.

For any integer $n\geq 1$, let $I^nF$ be the $n$-th power of the fundamental ideal $IF$ of $W(F)$ (we put $I^0F=W(F)$). Let $I^n_qF$ be the sub-group $I^{n-1}F\otimes W_q(F)$ of $W_q(F)$. This group is additively  generated by $n$-fold quadratic Pfister forms (see \cite[\S9.B]{Elman:2008}).

An $F$-quadratic form $\pi=(V,q)$ is called an $n$-fold quasi-Pfister form if $q(x)=B(x,x)$ for all $x\in V$, where $B$ is an $n$-fold bilinear Pfister form. In particular, quasi-Pfister forms are totally singular. A totally singular $F$-quadratic form $\sigma$ is called a quasi-Pfister neighbour if there exists an anisotropic quasi-Pfister form $\pi$ such that $2\dim \sigma>\dim \pi$ and $\sigma \prec_w \pi$. As with   Pfister neighbours,  in this case the form $\pi$ is unique, and for any field extension $K/F$, the form $\sigma_K$ is isotropic if and only if $\pi_K$ is isotropic and, in particular, the forms $\sigma_{F(\pi)}$ and $\pi_{F(\sigma)}$ are isotropic (see \cite[(8.9)]{hl}).

Two central simple $F$-algebras $A$ and $B$ are called Brauer-equivalent, denoted $A\sim B$, if they represent the same class in the Brauer group of $F$.
The degree of a central simple $F$-algebra $A$ is the integer $\sqrt{\dim_F A}$, and the index of $A$ is the integer $\sqrt{\dim_F D}$, where $D$ is the unique  central division $F$-algebra Brauer-equivalent to $A$.
A central simple algebra of degree two is know as a quaternion algebra.
For $a, b \in F$ with $b\neq 0$, we denote by $[a, b)$ the quaternion $F$-algebra whose standard $F$-basis $\{1, i, j, k\}$ satisfies the following relation:
$i^2+i=a$, $j^2=b$, $jij^{-1}=i+1$ and $k=ij$.

For $\phi$ an $F$-quadratic form, we denote by $C(\phi)$ (\resp $C_0(\phi)$) the Clifford algebra of $\phi$ (\resp the even Clifford algebra of $\phi$). If 
$\phi$ is nonsingular, then $C(\phi)$ is a central simple $F$-algebra, and the centre of $C_0(\phi)$ is a separable quadratic $F$-algebra $Z(\phi)$ (see \cite[\S11]{Elman:2008}). In this case, the Arf invariant of $\phi$, denoted $\triangle(\phi)$, is the class in the additive group $F/\wp (F)$ of an element $\delta \in F$ satisfying $Z(\phi)=F\left[ X\right]/(X^2+X +
\delta)$, where $\wp(F)=\{ a^2+a \mid a\in F\}$. In particular, if $\phi \cong [a_1,b_1]\perp \cdots \perp [a_r,b_r]$, then $\triangle(\phi)=a_1b_1+\cdots +a_rb_r+\wp(F)$ (see \cite[\S13]{Elman:2008}).

We will need  the following index reduction theorem:

\begin{theorem}\label{trind} Let $D$ be a  central simple division $F$-algebra and $\psi$ an $F$-quadratic form of dimension $\geq 2$.
\begin{enumerate}[$(1)$]
\item \cite[Th. 4]{ta} If $\psi$ is nonsingular and $\triangle(\psi)\neq 0$, then $D\otimes_F F(\psi)$ is not a division algebra if and only if $D$ contains a sub-algebra isomorphic to $C_0(\psi)$.
\item \cite[Th. 3]{ta} If $\psi=a_1[1,b_1]\perp \cdots \perp a_n[1,b_n]\perp \left<1, c_1, \cdots, c_m\right>$ is anisotropic of dimension $2n+m+1\geq 2$ with $m\geq 0$, then $D\otimes_F F(\psi)$ is not a division algebra if and only if $D$ contains a sub-algebra isomorphic to $[b_1,a_1)\otimes_F\cdots \otimes_F[b_n,a_n)\otimes F(\sqrt{c_1}, \cdots, \sqrt{c_m})$.
\end{enumerate}
\end{theorem}

We finish this section with some results needed in the proofs.

\begin{prop} \label{p1}
(\cite[Prop.~3.2]{l1}) Let $\phi=a[1,x] \perp \left< 1,b,c\right>$ be an anisotropic $F$-quadratic form. Then $\phi$ is a Pfister neighbour if and only if the algebra $[x,a)\otimes_FF(\sqrt{b},\sqrt{c})$ is split.
\end{prop}

\begin{prop} \label{p2}
Let $\phi=a[1,x] \perp \left< 1,c_1,\cdots, c_s\right>$ be an anisotropic $F$-quadratic form. Let $K/F$ be a field extension such that $i_W(\phi_K)=1$. Then  $[x,a)\otimes_FK(\sqrt{c_1}, \cdots, \sqrt{c_s})$ is split.
\end{prop}

\prf Let $L=K(\sqrt{c_1}, \cdots, \sqrt{c_s})$. Since $i_W(\phi_K)=1$ we have $\phi_K\cong [0,0]\perp \left< 1,c_1,\cdots, c_s\right>_K$. Then $a[1,x]_L\perp \left<1\right>_L \perp s\times \left<0\right> \cong [0,0] \perp \left<1\right>_L \perp s\times \left<0\right>$. By Proposition \ref{p3}(2), we deduce that 
$a[1,x]_L\perp \left<1\right>_L \cong [0,0] \perp \left<1\right>_L$, and thus, taking the even Clifford algebra \cite[Lem. 2]{ta}, we conclude that $[x,a)\otimes_FL$ is split.\qed

\vskip1mm

Quadratic forms of dimension $2^n+1$ satisfy many properties related to the isotropy problem over the function fields of quadrics. We recall two of these properties.

\begin{prop} \label{c1}
(\cite[Cor.~5.11]{l2})
Let $\phi$ and $\psi$ be anisotropic  $F$-quadratic forms of type $(1,s)$ and $(1,s')$ respectively. Suppose that $\dim \phi=2^n+1$ and $\dim \psi>2^n+1$ ($n\geq 1$). Then $\phi_{F(\psi)}$ is anisotropic.
\end{prop}

\begin{prop} \label{c2}
 (\cite[Th. 1.3]{hl1}) Let $\phi$ be an anisotropic $F$-quadratic form of dimension $2^n+1$ and $\psi$ an anisotropic totally singular $F$-quadratic form of dimension $>2^n$. Then $\phi$ is anisotropic over $F(\psi)$.
\end{prop}

We will need some facts about quadratic forms over valued fields. Let $K$ be a field of characteristic $2$ which is complete for a discrete valuation $\nu$, $A$ the associated valuation ring, $\pi$ a uniformiser   and $\overline{K}$ the residue field. Let $\phi=(V,q)$ be a $K$-quadratic form. The first and the second residue forms of $\phi$ are defined as follows (we refer to \cite[p.1341]{b1} for more details): For any integer $i\geq 0$, let $M_i=\{ v\in V\mid q(v)\in \pi^i A\}$. $M_i$ is  an $A$-module and  clearly $M_0\supset M_1 \supset M_2$. Let us consider the $\overline{K}$-vector spaces $V_0=M_0/M_1$ and $V_1=M_1/M_2$, and define the $\overline{K}$-quadratic forms $\phi_0=(V_0,q_0)$ and 
$\phi_1=(V_1,q_1)$ 
by $q_i(v+M_{i+1})=\overline{\pi^{-i}q(v)}$ for $i=0, 1$. The forms $\phi_0$ and $\phi_1$ are anisotropic and are called the first and the second residue forms of $\phi$, respectively. These forms may be singular and  they satisfy $\dim \phi=\dim \phi_0 +\dim \phi_1$. 

We recall the Schwarz inequality \cite[p.342]{mmw} which asserts that for any two vectors $x,y\in V$, we have:

\[\nu(b_{q}(x,y)^2) \geq \nu (q(x))+ \nu (q(y)),\]where $b_{q}$ is the polar form of $\phi$.

\begin{exa}\label{example1}
 We keep the same notations and hypotheses as in the previous paragraph. Let $u, v\in K$ be units and $n\in \Z$ be such that the binary quadratic form $[u,v\cdot \pi^n]$ is anisotropic over $K$. Then the Schwarz inequality implies that $n\leq 0$. If moreover $n <0$ and even (\resp $n <0$ and odd), then the first and the second residue forms of $[u,v\cdot \pi^n]$ are $\left< \overline{u}, \overline{v}\right>$ and the zero form (\resp $\left< \overline{u}\right>$ and $\left< \overline{v}\right>$). If $n=0$ then the first and  second residue forms are $[\overline{u}, \overline{v}]$ and the zero form.
\end{exa}

\section{Proof of Proposition \ref{ppar}}\label{ppart1}

Let $\phi$ be an anisotropic $F$-quadratic form  of dimension $5$ and type $(1,3)$ that is not a Pfister neighbour. Let $\psi$ be an anisotropic 
$F$-quadratic form of type $(r,s)$. 
Up to a scalar, we may write $\phi=x[1,a]\perp \left<1,b,c\right>$. Let $L=F(\sqrt{b}, \sqrt{c})$. Since $\phi$ is not a Pfister neighbour, the algebra $[a,x)\otimes_F L$ is division. If  $r=0$ and $s=1$ then clearly $\phi_{F(\psi)}$ is anisotropic.

(i) Suppose that ($r=0$ and $s\geq 5$) or ($r=1$ and $s\geq 4$), then $\phi_{F(\psi)}$ is anisotropic by Propositions \ref{c2} and \ref{c1}, respectively.

(ii) Suppose that $\psi=u[1,k]\perp v[1,l]$ and $\triangle(\psi)\neq 0$. If $\phi_{F(\psi)}$ is isotropic, then $i_W(\phi_{F(\psi)})=1$ because $\left<1,b,c\right>_{F(\psi)}$ is anisotropic. It follows from Proposition \ref{p2} that $[a,x)\otimes_F L(\psi_L)$ is not a division algebra. This implies that $\psi$ is anisotropic over $L$. Moreover, since $\triangle(\psi_L)\neq 0$, the algebra $[a,x)\otimes_FL$ contains a sub-algebra isomorphic to $C_0(\psi_L)$ by Theorem \ref{trind}(1). By comparing the dimensions of the two algebras, we see that this is not possible.

(iii) Suppose that $r=2$ and $s\geq 1$. Let $\psi'$ be an $F$-quadratic form dominated by $\psi$ of type $(2,1)$. If $\phi_{F(\psi)}$ is isotropic, then $\phi_{F(\psi')}$ is also isotropic because $F(\psi')(\psi)/F(\psi')$ is purely transcendental. Hence $\phi_{F(\psi')}$ is isotropic, and thus $\psi'$ is not a Pfister neighbour. This implies that $\psi'\cong R \perp \ql(\psi')$ for some nonsingular form $R$ such that $\triangle(R)\neq 0$. Now since $F(R)(\psi')/F(R)$ is purely transcendental, we conclude that $\phi_{F(R)}$ is isotropic, which is not possible by the case (ii).

(iv) Suppose that $r\geq 3$. Let $\rho$ be an $F$-quadratic form dominated by $\psi$ of type $(2,1)$. If $\phi_{F(\psi)}$ is isotropic then $\phi_{F(\rho)}$ is isotropic, which is not possible by the case (iii).

\section{Proof of statements $(1)$, $(2)$ and $(3)$ of Theorem \ref{tp}}

Let $\phi$ be an anisotropic $F$-quadratic form of dimension $5$ and type $(1,3)$ that is not a Pfister neighbour.
First let $\psi$ be an anisotropic form similar to a $2$-fold Pfister form $\pi$ over $F$ such that $\phi_{F(\psi)}$ is isotropic, and let $\psi'$ be a Pfister neighbour of $\pi$ of dimension $3$. Then $\phi_{F(\psi)}$ is isotropic if and only if $\phi_{F(\psi')}$ is isotropic, as
the field extensions 
$F(\psi')(\psi)/F(\psi')$  and $F(\psi)(\psi')/F(\psi)$  are transcendental. As $\psi'$ must be of type $(1,1)$, we have reduced  Case $(2)$ to Case $(1)$.

 Now let $\psi$ be an anisotropic $F$-quadratic form  of type $(1,s)$ with $1\leq s \leq 3$ and assume that  $\phi_{F(\psi)}$ is isotropic.
Up to a scalar we may write $\psi\cong R_2 \perp \ql(\psi)$, where $R_2$ is nonsingular of dimension $2$ and $\ql(\psi)$ is one of the following forms $\left<1\right>$, $\left< 1,a\right>$ or $\left< 1, a,b\right>$  as $s=1$, $2$ or $3$, accordingly. Similarly, up to a scalar, we may write $\phi\cong R_1 \perp \left<1,u,v\right>$. Let $Q_1$ and $Q_2$ be the quaternion $F$-algebras satisfying $C(R_i)\sim Q_i \in {\rm Br}(F)$ for $i=1, 2$. Since $\phi_{F(\psi)}$ is isotropic, the algebra $Q_1\otimes_FF(\psi)(\sqrt{u},\sqrt{v})$ is split (Proposition \ref{p2}).

\noindent{\bf Claim 1}: $\ql(\psi)$ is similar to a subform of $\ql(\phi)$. 

This is  trivial if  $s=1$. Suppose that $s\geq 2$. We have $F(\psi)(\sqrt{u},\sqrt{v})=F(\sqrt{u},\sqrt{v})(\psi)$ as the nondefective part of $\psi_{F(\sqrt{u},\sqrt{v})}$ is neither of type $(0,1)$ nor isometric to $\H$. 

Since $Q_1\otimes_FF(\sqrt{u},\sqrt{v})(\psi)$ is split, we conclude by Theorem \ref{trind}(2) that $\psi$ is necessarily isotropic over $F(\sqrt{u}, \sqrt{v})$. The case $i_W(\psi_{F(\sqrt{u}, \sqrt{v})})>0$ is excluded otherwise $F(\sqrt{u},\sqrt{v})(\psi)$ would be purely transcendental over $F(\sqrt{u},\sqrt{v})$ and thus $Q_1\otimes_FF(\sqrt{u},\sqrt{v})$ would be split. Hence $i_d(\psi_{F(\sqrt{u}, \sqrt{v})})>0$. Moreover, by reasons of dimension, Theorem \ref{trind}(2) implies that $\dim (\ql(\psi)_{F(\sqrt{u}, \sqrt{v})})_{an}=1$. Hence when $s=2$ (\resp $s=3$) this implies that $a\in F^2(u,v)$ (\resp $a,b\in F^2(u,v)$). Consequently, $\left<\left<u,v\right>\right>$ is isotropic over $F(\left<1,a\right>)$ or $F(\left<1,a,b\right>)$  as $s=2$ or $s=3$, accordingly. In particular, $\left< 1,u,v\right>$ is isotropic over $F(\left<1,a\right>)$ or $F(\left<1,a,b\right>)$  as $s=2$ or $s=3$, accordingly. Hence $\ql(\psi)$ is similar to a subform of $\left<1,u,v\right>$ (the case $s=2$ is Lemma \ref{easy} and the case $s=3$ is a consequence of \cite[Thm.~1.2]{lr}). Hence, up to a scalar, we may suppose that $\ql(\psi)$ is a subform of $\left<1,u,v\right>$. 

By Claim 1  the nondefective part of $\psi_{F(\sqrt{u}, \sqrt{v})}$ is isometric to $(R_2\perp \left<1\right>)_{F(\sqrt{u}, \sqrt{v})}$. 
Since $Q_1\otimes_F F(\sqrt{u}, \sqrt{v})(\psi)$ is split, it follows that $Q_1\otimes_F F(\sqrt{u}, \sqrt{v})(R_2\perp \left<1\right>)$ is also split. Consequently, Theorem \ref{trind}(2) implies that $Q_1\otimes_F F(\sqrt{u}, \sqrt{v})$ is isomorphic to $Q_2\otimes_F F(\sqrt{u}, \sqrt{v})$. This implies that $Q_1\otimes_FQ_2$ is split over $F(\sqrt{u}, \sqrt{v})$. Then there exist $k,l\in F$ such that $Q_1\otimes_FQ_2\sim [k,u)+[l,v) \in {\rm Br}(F)$. Using the Clifford invariant, we get  
$$R_1\perp R_2 \perp u[1,k]\perp v[1,l]\perp [1,r_1+r_2+u+v] \in I^3_qF\,,$$where $\triangle(R_i)=r_i+\wp(F)$ for $i=1,2$. Hence, by \cite[Prop. 6.4]{l3}, there exists $\pi \in GP_3F$ such that
\begin{equation}
R_1\perp R_2 \perp \rho \sim \pi,
\label{e0p}
\end{equation}where $\rho=u[1,k]\perp v[1,l]\perp [1,r_1+r_2+u+v]$ is a nonsingular completion of $\left< 1,u,v\right>$.\vskip1mm

\noindent{\bf Claim 2}: The form $\pi$ is isotropic over $F(\psi)$. This implies by Theorem \ref{trd} that $\psi\prec_w \pi$.

Since $\psi_{F(\psi)}$ is isotropic, we get $(R_2\perp \left<1,u,v\right>)_{F(\psi)}\cong \H \perp \left<1,u,v\right>)_{F(\psi)}$. By the completion lemma (Proposition \ref{p6}), there exist $r,s,t\in F(\psi)$ such that
$$R_2 \perp \rho\cong \H  \perp u[1,r]\perp v[1,s]\perp [1,t]\,.$$
Hence we get
\begin{equation}
R_1\perp R_2 \perp \rho \sim R_1 \perp u[1,r]\perp v[1,s]\perp [1,t] \sim \pi.
\label{e1p}
\end{equation}

Since $\phi_{F(\psi)}$ is isotropic, we get $(R_1\perp \left<1,u,v\right>)_{F(\psi)}\cong [0,0]\perp \left<1,u,v\right>)_{F(\psi)}$. Again by the completion lemma, there exist $r',s',t'\in F(\psi)$ such that 
\begin{equation}
R_1 \perp u[1,r]\perp v[1,s]\perp [1,t]\cong \H \perp u[1,r']\perp v[1,s']\perp [1,t'].
\label{e2p}
\end{equation}
It follows from (\ref{e1p}) and (\ref{e2p}) that $u[1,r']\perp v[1,s']\perp [1,t'] \sim \pi_{F(\psi)}$, and thus $\pi_{F(\psi)}$ is isotropic. Hence the claim.\vskip1mm

\noindent{\bf Claim 3}: If $s=3$, then $\phi$ is isometric to $\psi$.

Without loss of generality, we may suppose that $\pi \in P_3F$. If $\pi$ is anisotropic then $\psi$ is a Pfister neighbour of $\pi$. Hence $\psi_{F(\pi)}$ is also isotropic and $F(\pi)(\psi)/F(\pi)$ is purely transcendental. Consequently $\phi$ is isotropic over $F(\pi)$, which is not possible by Proposition \ref{ppar}. Hence $\pi$ is isotropic and thus hyperbolic. It follows from (\ref{e0p}) that $R_1 \perp \rho \sim R_2$. Hence $R_1 \perp \rho \perp \left< 1,u,v\right>\sim R_2 \perp \left< 1,u,v\right>$. Consequently, $R_1 \perp \left< 1,u,v\right>\sim R_2 \perp \left< 1,u,v\right>$, which implies that $\phi$ is isometric to $\psi$. 

Conversely, suppose that there exist $R_1, R_2$ nonsingular quadratic forms of dimension $2$, scalars $u,v, \alpha, \beta \in F^*$, a nonsingular completion $\rho$ of $\left<1,u,v\right>$ and a Pfister form $\pi\in P_3F$ such that: $\psi \prec_w \pi$, $\alpha \phi \cong R_1 \perp \left<1,u,v\right>$, $\beta\psi\cong R_2\perp Q$ and $R_1\perp R_2 \perp \rho \sim \pi$ such that $Q$ is a subform of $\left<1,u,v\right>$. Then $(R_1\perp R_2 \perp \rho)_{F(\psi)} \sim 0$. In particular, 
$(R_1\perp \rho \perp \left<1, u,v\right>)_{F(\psi)} \sim (R_1\perp \left<1, u,v\right>)_{F(\psi)} \sim (R_2\perp \left<1,u,v\right>)_{F(\psi)}$. Since 
$(R_2\perp \left<1,u,v\right>)_{F(\psi)}$ is isotropic, we conclude that $\phi_{F(\psi)}$ is isotropic.

\section{Proof of statements $(4)$ and $(5)$ of Theorem \ref{tp}}

Let $\phi$ be an anisotropic $F$-quadratic form of dimension $5$ and type $(1,3)$ that is not a Pfister neighbour and  $\psi$ be an anisotropic $F$-quadratic form  of type $(0,3)$ or $(0,4)$.  Suppose that $\phi_{F(\psi)}$ is isotropic. 
We want to show that there exist  a quadratic form $\phi'$ of type $(1,3)$ and $\pi\in GP_3(F)$ such that $\phi \sim \phi' \perp \pi$ and $\psi$ is weakly dominated by $\phi'$ and $\pi$.
Up to a scalar, we may suppose that $\phi=\alpha[1,x]\perp \left< 1,u,v\right>$ for suitable $\alpha, x, u,v\in F^*$.

\noindent{\bf Case 1.} Suppose that $\psi$ is of type $(0,3)$. Up to a scalar, we may suppose that $\psi=\left< 1,a,b\right>$. We put $\delta=\left< 1,u,v\right>$.

(a) Suppose that $\delta$ is isotropic over $F(\psi)$. Then $\delta$ is similar to $\psi$ by \cite[Thm.~1.2]{lr}. Hence we are done by taking $\phi'=\phi$ and $\pi$ the hyperbolic $3$-fold quadratic Pfister form.

(b) Suppose that $\delta$ is anisotropic over $F(\psi)$.

\noindent{\it Claim.} Up to a scalar, we may suppose that $a$, $b$ or $ab$ is represented by $\delta$.

Since $\delta$ is anisotropic over $F(\psi)$, the isotropy of $\phi_{F(\psi)}$ implies that $i_W(\phi_{F(\psi)})=1$, i.e. $\phi_{F(\psi)}\cong \H \perp \delta_{F(\psi)}$. Moreover, the isotropy of $\delta$ over its own function field implies that $\delta_{F(\delta)}\cong \left<1, u\right>_{F(\delta)}\perp \left< 0\right>$. Hence
$$\phi_{F(\delta)(\psi)}\sim (\alpha[1,x]\perp \left<1,u\right> \perp \left<0\right>)_{F(\delta)(\psi)} \sim (\H\perp \left<1,u\right> \perp \left<0\right>)_{F(\delta)(\psi)}\,.$$Since the forms $(\alpha[1,x]\perp \left<1,u\right>)_{F(\delta)(\psi)}$ and $(\H\perp \left<1,u\right>)_{F(\delta)(\psi)}$ are nondefective, it follows from Proposition \ref{p3} that$$(\alpha[1,x]\perp \left<1,u\right>)_{F(\delta)(\psi)} \sim (\H\perp \left<1,u\right>)_{F(\delta)(\psi)}\,.$$ In particular, the form $(\alpha[1,x]\perp \left<1,u\right>)_{F(\delta)(\psi)}$ is isotropic. 

Note that the form $(\alpha[1,x]\perp \left<1,u\right>)_{F(\delta)}$ is anisotropic, otherwise we would get that the algebra $[x,\alpha)\otimes_F F(\delta)(\sqrt{u})$ is split, and thus $[x,\alpha)\otimes_F F(\sqrt{u}, \sqrt{v})$ would be split, i.e. $\phi$ would be a Pfister neighbour. Hence the Albert form $\gamma:=\alpha[1,x]\perp u[1,t^{-1}]\perp [1,x+t^{-1}]$ is anisotropic over $K(\delta)$, where $K=F((t))$ the field of Laurent series. Hence $D:=[x,\alpha)_K\otimes_K [t^{-1},u)$ is a division algebra over $K(\delta)$. Since $(\alpha[1,x]\perp \left<1,u\right>)_{F(\delta)(\psi)}$ is isotropic, it follows that $\gamma$ is isotropic over $K(\delta)(\psi)$, and thus the index of the algebra $D:=[x,\alpha)_K\otimes_K [t^{-1},u)$ reduces over the extension $K(\delta)(\psi)$. By Theorem \ref{trind}(2), $D_{K(\delta)}$ contains the biquadratic extension $K(\delta)(\sqrt{a}, \sqrt{b})$. This implies that the algebra $D_{K(\delta)}$ is isomorphic to the biquaternion algebra of $[r,a)\otimes [s,b)$ for suitable $r,s\in K(\delta)^*$. Hence, using the Jacobson theorem \cite{ms}, there exists $p\in K(\delta)^*$ such that
\begin{equation}
(\alpha[1,x]\perp u[1,t^{-1}]\perp [1,x+t^{-1}])_{K(\delta)}\cong p(a[1,r] \perp b[1,s] \perp [1,r+s]).
\label{eqiso1}
\end{equation}

We consider the $t$-adic valuation of the field $K(\delta)$. It is clear that, from the left hand side of (\ref{eqiso1}), the first and second residue forms of $\gamma_{K(\delta)}$ are the forms $\alpha[1,x]\perp \left<1, u\right>$ and $\left<1,u\right>$ respectively. We may suppose that $p$ is square free. Moreover $p$ is a unit, otherwise the second residue form from the right hand side of (\ref{eqiso1}) would be of dimension bigger that $2$. Using the Schwarz inequality we deduce that the valuations of $r$, $s$ and $r+s$ are less than or equal to zero (due  to the anisotropy of $\gamma$ over $K(\delta)$, see Example \ref{example1}). Moreover, using Example \ref{example1} and the fact that the first and the second residue forms of the left hand side of (\ref{eqiso1}) are of type $(1,2)$ and $(0,2)$, we conclude that one of the scalars $r,s$ and $r+s$ is a unit and the two other scalars are not units whose valuations are odd. Hence comparing the quasilinear parts of the residue forms, we get that $\left<1,u\right>_{F(\delta)}$ is isometric to one of the following forms: $\overline{p}\left<1,a\right>_{F(\delta)}$, $\overline{p}\left<1,b\right>_{F(\delta)}$ or $\overline{p}\left<a,b\right>_{F(\delta)}$, where $\overline{p}$ is the residue class of $p$. Using the roundness of a quasi-Pfister form (\cite[(8.5), (i)]{hl}), we conclude that 
\[(\star)\;\;\;\;\;\left<1,u\right>_{F(\delta)}\cong \begin{cases}\left<1,a\right>_{F(\delta)}\; \text{or}\\\left<1,b\right>_{F(\delta)}\; \text{or}\\
\left<1,ab\right>_{F(\delta)},\end{cases}\]which implies that one of the following three forms is isotropic over $F(\delta)$: $\left<1,u,a\right>$, $\left<1,u,b\right>$, $\left<1,u,ab\right>$. Using \cite[Thm.~1.2]{lr}, and modulo a scalar, we may therefore suppose that $\delta$ is isometric to one of the three forms: $\left<1,u,a\right>$, $\left<1,u,b\right>$, $\left<1,u,ab\right>$. Hence the claim.

By the claim above, we may suppose that $\psi=\left< 1,u,w\right>$ for suitable $w\in F^*$. As before the isotropy of $\phi_{F(\psi)}$ implies that $[x,\alpha)\otimes F(\sqrt{u}, \sqrt{v}, \sqrt{w})$ is split. Hence there exist suitable scalars $k,l, m\in F^*$ such that $[x, \alpha)$ is Brauer-equivalent to $[k, u)\otimes_F[l,v)\otimes_F[m,w)$. Using the Clifford invariant, we get that$$\alpha[1,x]\perp u[1,k] \perp v[1,l]\perp w[1,m]\perp [1,x+k+l+m]\in I^3_q(F)\,.$$ 
It follows from \cite[Prop. 6.4]{l3} that
\begin{equation}
\alpha[1,x]\perp u[1,k] \perp v[1,l]\perp w[1,m]\perp [1,x+k+l+m] \sim \pi
\label{eqiso2}
\end{equation}
for some form $\pi \in GP_3(F)$. Using the fact that $\phi_{F(\psi)}\sim \left<1,u,v\right>_{F(\psi)}$ with the completion lemma, we deduce from (\ref{eqiso2}) that $(\alpha[1,x]\perp u[1,k] \perp v[1,l]\perp [1,x+k+l+m])_{F(\psi)} \cong \H \perp u[1,k'] \perp v[1,l']\perp [1,m']$ for suitable $k',l',m' \in F(\psi)^*$. Hence $u[1,k'] \perp v[1,l']\perp w[1,m] \perp [1,m'] \cong \pi_{F(\psi)}$. In particular, $\psi_{F(\psi)}$ is dominated by $\pi_{F(\psi)}$. This implies that $\pi_{F(\psi)}$ is isotropic, and thus hyperbolic.  Hence $\psi\prec_w\pi$. Further,  since$$\alpha[1,x]\perp u[1,k] \perp v[1,l]\perp w[1,m]\perp [1,x+k+l+m] \sim \pi\,,$$we deduce that 
$$\alpha[1,x]\perp \left<1,u,v\right> \sim w[1,m]\perp \left<1,u,v\right> \perp \pi\,.$$So we take $\phi'=w[1,m]\perp \left<1,u,v\right>$ which dominates $\psi$. 

Conversely, if there exist $\phi'$ of type $(1,3)$ and $\pi\in GP_3(F)$ such that $\phi\sim \phi' \perp \pi$, $\psi$ is weakly dominated by $\phi'$ and $\pi$, then $\phi_{F(\psi)} \sim \phi'_{F(\psi)}$, and thus $\phi_{F(\psi)}$ is isotropic.

\noindent{\bf Case 2.} Suppose that $\psi$ is of type $(0,4)$ and $\nd_F(\psi)=8$. We will apply the previous case (i.e., the case of type $(0,3)$) several times. Let  $\psi'$ be a subform of $\psi$ of dimension $3$. Since $\phi_{F(\psi)}$ and $\psi_{F(\psi')}$ are isotropic, we get that $\phi_{F(\psi')}$ is isotropic  by \cite[Lemme 4.5]{l1}. By the claim in the previous case, we may suppose that, up to a scalar, $\phi\cong \alpha[1,x] \perp \left<1,u,k\right>$ for suitable $u,k \in F^*$, and $\psi'=\left<1, u, b\right>$. So we write $\psi=\left<1,u,b,c\right>$. We put $\delta=\left<1,u,k\right>$. Now we repeat the same argument for the form $\psi''=\left<1,b,c\right>$. We conclude as in $(\star)$ that
\[\left<1,u\right>_{F(\delta)}\cong \begin{cases}\left<1,b\right>_{F(\delta)}\; \text{or}\\\left<1,c\right>_{F(\delta)}\; \text{or}\\
\left<1,bc\right>_{F(\delta)}.\end{cases}\]

(a) The first two possibilities give that $\left< 1,u,b\right>$ or $\left<1, u,c\right>$ is isotropic over $F(\delta)$, and thus by \cite[Thm.~1.2]{lr} this implies that $\delta$ is similar to $\left< 1,u,b\right>$ or $\left<1, u,c\right>$.

(b) The third possibility gives that $\left<1,u,bc\right>$ is isotropic over $F(\delta)$. The form $\left<1,u,bc\right>$ is anisotropic, otherwise we would get that $\nd_F(\psi)=4$. Again by \cite[Thm.~1.2]{lr} we conclude that $\delta$ is similar to $\left< 1,u,bc\right>$. Hence, up to a scalar, we may suppose that $\phi\cong \alpha[1,x] \perp \left<1,u,bc\right>$ and $\delta= \left<1,u,bc\right>$. Now we consider the form $\eta=\left<1, ut^2+b,c\right>$.  We know that $F(t)(\eta)$ is isometric to $F(\psi)$. Hence $\phi_{F(t)(\eta)}$ is isotropic. Again we reproduce the same argument as in $(\star)$ in {Case 1} to conclude that \[\left<1,u\right>_{F(t)(\delta)}\cong \begin{cases}\left<1,ut^2+b\right>_{F(t)(\delta)}\; \text{or}\\\left<1,c\right>_{F(t)(\delta)}\; \text{or}\\
\left<1,uct^2+bc\right>_{F(t)(\delta)}.\end{cases}\]

(b.1) In the first possibility, we conclude that $\left<1, u, ut^2+b\right>_{F(t)(\delta)}\cong \left<1,u, b\right>_{F(t)(\delta)}$ is isotropic. It follows from \cite[Thm.~1.2]{lr} that $\delta$ is similar to $\left< 1, u,b\right>$.

(b.2) In the second possibility, we conclude as in case (b.1) that $\delta$ is similar to $\left< 1, u,c\right>$.

(b.3) In the third possibility, we conclude that $\left<1,u\right>_{F(t)(\delta)}$ represents $uct^2+bc$. Since $\left<1,u\right>_{F(t)(\delta)}$ represents $bc$ (since  $\delta$ is isotropic over its own function field), it follows that $\left<1,u\right>_{F(t)(\delta)}$ represents $uct^2$, and in particular it represents $uc$. Hence $\left<1, u, uc\right>_{F(\delta)}$ is isotropic. Consequently $\left<1, u, c\right>_{F(\delta)}$ is isotropic because $\left<1, u, uc\right>$ and $\left<1, u, c\right>$ are quasi-Pfister neighbours of the same quasi-Pfister form $\left<\left<u,c\right>\right>$. Hence we get by \cite[Thm.~1.2]{lr} that $\delta$ is similar to $\left<1,u,c\right>$. 

By  cases (a) and (b), we may suppose, up to a scalar, that $\phi \cong \alpha[1,x] \perp \left<1,u,v\right>$ and $\psi=\left<1,u,v,w\right>$ for some $w\in F$. 

The isotropy of $\phi_{F(\psi)}$ implies that $[x,\alpha)\otimes F(\sqrt{u}, \sqrt{v}, \sqrt{w})$ is split. Now we follow the same argument as in Case 1 to conclude the existence of a form $\phi'$ of type $(1,3)$, a form $\pi \in GP_3(F)$ such that $\phi \sim \phi' \perp \pi$ and $\psi$ is weakly dominated by $\phi'$ and $\pi$. Conversely, these condition give the isotropy of $\phi_{F(\psi)}$ as proved in Case 1.

\noindent{\bf Case 3.} Suppose that $\psi$ is of type $(0,4)$ and $\nd_F(\psi)=4$. Let $\psi'$ be a subform of $\psi$ of dimension $3$. Since $\psi$ and $\psi'$ are quasi-Pfister neighbour of the same quasi-Pfister form, it follows that $\phi_{F(\psi)}$ is isotropic if and only if $\phi_{F(\psi')}$ is isotropic. Hence we have reduced this case to Case 1.

\end{document}